\documentclass{siamltex}
\title{Extended block Hessenberg  process for the evaluation of matrix functions}

\author{
A. H. Bentbib\thanks{Laboratory LAMAI, Faculty of Sciences and Technologies, Cadi Ayyad
University, Marrakech, Morocco. E-mail: a.bentbib@uca.ac.ma}
\and M. El Ghomari\thanks{Department of Mathematics, Ecole Normale Supérieure, Mohammed V University in Rabat, Av. Mohammed Belhassan El Ouazzani, Takaddoum, B.P.
5118, Rabat, Morocco. E-mail: m.elghomari10@gmail.com}
\and K. Jbilou\thanks{Laboratory LMPA, 50 Rue F. Buisson, ULCO Calais cedex, France, and
Laboratory CSEHS, University UM6P, Bengu\'erir, Morocco.
E-mail: Khalide.Jbilou@univ-littoral.fr}
\and El. M. Sadek\thanks{Laboratory of Engineering Sciences for Energy, National School of Applied Sciences of El Jadida, University Chouaib Doukkali, Morocco.  E-mail: sadek.e@ucd.ac.ma,  sadek.maths@gmail.com}
}

\usepackage{graphicx,epstopdf,epsfig}
\usepackage{amsfonts,epsfig,graphics,amsmath,amssymb}
\usepackage{color}
\usepackage{algorithm,algorithmic}
\usepackage{ulem}
\usepackage{algorithm,algorithmic}
\usepackage{multirow}
\usepackage{subcaption}

\newcommand{\R}{{\mathbb R}}

\newcommand{\E}{\widetilde{E} }
\newcommand{\e}{\widetilde{e} }
\newcommand{\Eh}{\widehat{E} }
\newcommand{\s}{\sum\limits }
\newcommand{\V}{\widetilde{V} }
\newcommand{\Vm}{\mathbb{V}}

\newcommand{\lu}{{\rm lu}}
\newcommand{\Hm}{\widetilde{\mathbb{H}}}
\newcommand{\Tm}{\mathbb{T}_{2m}}
\newcommand{\Sm}{\mathbb{S}_{2m}}

\begin{document}
	
%\subtitle{Do you have a subtitle?\\ If so, write it here}
%\titlerunning{Short form of title}        % if too long for running head
	
\maketitle

\begin{abstract}
In the present paper, we propose a block variant of the extended Hessenberg process for computing approximations of matrix functions and other problems producing large-scale matrices. Applications to the computation of a matrix function such as $f(A)V$, where $A$ is an $n\times n$ large sparse matrix, $V$ is an $n\times p$ block with $p\ll n$, and $f$ is a function are presented. Solving  shifted linear systems with multiple right hand sides are also given. Computing approximations of these matrix problems appear in many scientific and engineering applications. Different numerical experiments are provided to show the effectiveness of the proposed method for these problems.
\end{abstract}


\section{Introduction}\label{Intro}
Let $A\in\R^{n\times n}$ be a large and sparse matrix, and let $V\in\R^{n\times p}$ with $p\ll n$. We are interesting in approximating numerically expressions of the form
\begin{align}\label{If}
    \mathbb{I}(f):=f(A)V,
\end{align}

where $f$ is a function that is defined on the convex hull of the spectrum of $A$. Evaluation of expressions \eqref{If} arises in various applications such as in network analysis when $f(t)=exp(t)$  \cite{Estrada}, in machine learning when $f(t)=log(t)$ \cite{NBS,HAN}, in quantum chromodynamics when computing Schatten p-norms, when $f(t)=t^{p/2}$, for some $0<p\leq 1$;  \cite{BGMSUX,SCS}, and in the solution of ill-posed problems \cite{FRRS,Hansen}. The matrix function $f(A)$ can be defined by the spectral factorization of $A$ (if exists) or using other techniques; see, e.g., \cite{Higham} for discussions on several possible definitions of matrix functions.  In many applications, the matrix $A$ is so large, the computation of $f(A)$ is not feasible. Several projection methods have been developed \cite{AHJ,Bentbib,Bentbib1,Datta,Datta2,HJ,HJMT,FJ}. These projection methods are based on the variants of standard (or block) Arnoldi and Lanczos techniques by using  rational, or extended Krylov subspaces. In the context of approximating the expressions of the form $f(A)b$ for some vector $b\in\R^n$, Druskin et al. \cite{DK,KS} have proposed the extended Arnoldi process when $A$ is nonsingular. This subspace is determined by both positive and negative powers of $A$. It was shown that the approximation of $\eqref{If}$ using the extended Krylov subspaces is more accurate than the approximation using polynomial Krylov subspaces.

\noindent We also concerned with the solution of shifted linear systems with multiple right hand sides of the form
\[
(A+\sigma I_n)X^{\sigma}=C,
\]
which needs to  be solved for many values of $\sigma$, where $C\in\mathbb{R}^{n\times p}$. The shifts are distinct from the eigenvalues of $A$. The solution $X^{\sigma}$ may be written as $X^{\sigma}:=(A+\sigma I)^{-1}C\equiv f(A)C$, with $f(z)=(z+\sigma)^{-1}$ is the resolvent function. In this paper, we present the extended block Hessenberg method with pivoting strategy, for approximating $f(A)V$ and also for solving linear systems with multiple right hand sides. The method presented generalizes the Hessenberg method with pivoting strategy discussed in \cite{AHS,ATG,RT}, which use (standard) block Krylov subspaces, to allow the application of extended block Krylov subspaces. The latter spaces are the union of a (standard) block Krylov subspace determined by positive powers of $A$ and a block Krylov subspace defined by negative powers of $A$.

 \noindent The rest of the paper is organized as follows. In the next section, we  introduce the extended block Hessenberg process with some properties. In section $3$, we describe the application of this process to the approximation of the matrix functions of the form \eqref{If}. The solution of the shifted linear system with multiple right hand sides by using the proposed method is presented in Section $4$. Finally, Section $5$ is devoted to
numerical experiments.

\section{The extended block Hessenberg process}\label{section:ebhp}

Let $A\in\mathbb{R}^{n\times n}$ be the nonsingular matrix and the block vector $V\in\R^{n\times p}$. The extended block Krylov subspace $\mathbb{K}^e_m(A,V)$ is the subspace of $\mathbb{R}^n$ generated by the columns of the blocks $A^{-m}V,\ldots,A^{-1}V,V,\ldots,A^{m-1}V$. This subspace is defined by
\begin{equation}\label{Kme}
\mathbb{K}^e_m(A,V)={\rm range}\{V,A^{-1}V\ldots,A^{m-1}V,A^{-m}V\}\subset \R^{n\times s}.
\end{equation}
The extended block subspace can be considered as a sum of two block Krylov subspaces. The first one is related to the pair $(A,V)$ while the second one is related to $(A^{-1},A^{-1}V)$.

Similar to the classical Hessenberg process with pivoting strategy \cite{Wilkinson}, the extended block Hessenberg method generates a unit lower trapezoidal basis  $\mathbb{V}_{2m}=\{V_1,\ldots, V_{2m}\}\in\R^{n\times 2mp}$ for the extended block subspace $\mathbb{K}^e_m(A,V)$, with the $V_i$'s are block vectors of size $n\times p$. The  $[v_1,\ldots,v_m]\in\R^{n\times m}$ is said to be  \cite{Sadok} unit lower trapezoidal matrix, if the first $k-1$ components of $v_k$ equal to zero and the $k$-th component of $v_k$ equal to one. We apply a LU decomposition with  partial pivoting (PLU decomposition) of the block vector $V$, we obtain $P_1V=L_1\Gamma_{1,1}$, where $P_1\in\R^{n\times n}$ is a permutation matrix, $L_1\in\R^{n\times p}$ is a unit trapezoidal matrix and $\Gamma_{1,1}\in\R^{p\times p}$ is an upper triangular matrix. Then the first block vector $V_1$ can be obtained as follows
\begin{align}\label{ComputeV1V2}
    V_1=P^T_1L_1=V\Gamma_{1,1}^{-1}.
\end{align}

\noindent $\Gamma_{1,1}$ and $V_1$ can be computed using the $\lu$ Matlab function $[V_1,\Gamma_{1,1}]=\lu(V).$  The $\lu$ Matlab function  applied to the matrix $V$  returns a permuted lower triangular matrix $V_1$ and an upper triangular matrix $\Gamma_{1,1}$ such that $V=V_1\Gamma_{1,1}$. 

\noindent The block vector $V_2$ is obtained as follows
\begin{align}\label{V2}
    \widetilde{V}_2=A^{-1}V-V_1\Gamma_{1,2}, \quad V_2=P^T_2L_2=\widetilde{V}_2\Gamma_{2,2}^{-1},
\end{align}
where the matrices $P_2\in\R^{n\times n},$ $L_2\in\R^{n\times p}$ and $\Gamma_{2,2}\in\R^{p\times p}$ are obtained by applying the PLU decomposition to $\widetilde{V}_2.$ $\Gamma_{2,2}$ and $V_2$ are computed by using  $[V_2,\Gamma_{2,2}]=\lu(\widetilde{V}_2).$  Let $[p_1,p_2]=\{i_1,\ldots,i_{2p}\}$ be the set of indices $i_j$, with $i_j$ is the index of the row of $[V_1,V_2]$ corresponding to the $j$-th row of $[L_1,L_2]$, where $p_1^T,p_2^T\in\R^p$ and let $[\E_1,\E_2]=[\e_{i_{1}},\ldots,\e_{i_{2p}}]$, where $\e_i$ is the $i$-th vector of the canonical basis of $\R^n$. $\E_1$ and $\E_2$ are made up of the first $p$ columns of $P_1$ and $P_2$; respectively.  $\Gamma_{1,2}\in\R^{p\times p}$ is determined so that $\widetilde{V}_2\bot\widetilde{E}_1$. Thus,
\begin{align}\label{gamma12}
\Gamma_{1,2}=\widetilde{E}_1^TA^{-1}V=V_1(p_1,:)^{-1}(A^{-1}V)(p_1,:).
\end{align}

\noindent To compute the block vectors $V_{2j+1}$ and $V_{2j+2}$ for $j=1,\ldots,m$, we use the following formulas
\begin{equation}\label{blockvectorsHessenberg}
\begin{array}{ccccccll}
\V_{0,2j+1}&=&AV_{2j-1},& \text{ and } &\V_{i,2j+1}&=&AV_{2j-1}-\sum\limits_{k=1}^{i}V_kH_{k,2j-1},& i=1,\ldots,2j,\\
\V_{0,2j+2}&=&A^{-1}V_{2j},& \text{ and } &\V_{i,2j+2}&=&A^{-1}V_{2j}-\sum\limits_{k=1}^{i}V_kH_{k,2j}& i=1,\ldots,2j+1,
\end{array}
\end{equation}
where, the $p\times p$ square matrices $H_{1,2j-1},\ldots,H_{2j,2j-1}$ and $H_{1,2j},\ldots,H_{2j+1,2j}$ are determined so that
\begin{align}\label{Galerkin}
\V_{2j+1}\bot\, \E_1,\ldots,\E_{2j},\text{ and } \V_{2j+2}\bot\, \E_1,\ldots,\E_{2j+1}.
\end{align}
Thus,  $H_{1,2j-1},\ldots,H_{2j,2j-1}$ and $H_{1,2j},\ldots,H_{2j+1,2j}$ are written as

\begin{equation}\label{HHH}
\begin{array}{lcr}
  H_{k,2j-1}&=&\widetilde{E}_k^T\V_{k-1,2j+1}=(V_k(p_k,:))^{-1}\V_{k-1,2j+1}(p_k,:),\\
  H_{k,2j}&=&\widetilde{E}_k^T\V_{k-1,2j+2}=(V_k(p_k,:))^{-1}\V_{k-1,2j+2}(p_k,:),
\end{array}
\end{equation}

where $p_{k}=i_{p(k-1)+1},\ldots,i_{pk}$ be the vector of indices $i_j$, with $i_j(j=p(k-1)+1,\ldots,kp)$ is the index of the row of $V_k$ corresponding to the $j$-th row of $L_k$ and $\E_k=[\e_{i_{p(k-1)+1}},\ldots,\e_{i_{pk}}].$

We apply the PLU decomposition to $\V_{2j,2j+1},$ we get $P_{2j+1}\V_{2j,2j+1}=L_{2j+1}H_{2j+1,2j-1}$. Then the block vector $V_{2j+1}$ is
\[
V_{2j+1}=P^T_{2j+1}L_{2j+1}=\V_{2j,2j+1}H^{-1}_{2j+1,2j-1}.
\]

We apply again the PLU decomposition to $\V_{2j+1,2j+2},$ we get $P_{2j+2}\V_{2j+1,2j+2}=L_{2j+2}H_{2j+2,2j}$. Then the block vector $V_{2j+2}$ is
\[
V_{2j+2}=P^T_{2j+2}L_{2j+2}=\V_{2j+1,2j+2}H^{-1}_{2j+2,2j}.
\]

\noindent The pairs $(V_{2j+1},H_{2j+1,2j-1})$ and $(V_{2j+2},H_{2j+2,2j})$ are obtained by using $\lu$ Matlab function, i.e.,
\[
[V_{2j+1},H_{2j+1,2j-1}]=\lu(\widetilde{V}_{2j,2j+1}),\quad [V_{2j+2},H_{2j+2,2j}]=\lu(\widetilde{V}_{2j+1,2j+2}).
\]

\begin{algorithm}
\caption{The extended block Hessenberg algorithm with pivoting strategy  $(EBHA)$}\label{algorithm:EBHA}
\textbf{Inputs:} Nonsingular matrix $A$, initial block $V$, and an integer $m$.
\begin{enumerate}
    \item $[V_1,\Gamma_{1,1}]=\lu(V)$; $[\sim,p_1]=\max{(V_1)}$;
    \item $\Gamma_{1,2}=V_1(p_1,:)^{-1}A^{-1}V(p_1,:);$\quad $\widetilde{V}_2=A^{-1}V-V_1\Gamma_{1,2};$
    \item $[V_2,\Gamma_{2,2}]=\lu(\widetilde{V}_2)$; $[\sim,p_2]=\max{(V_2)}$;
    \item For $j=1:m$
    \begin{enumerate}
        \item $\V=AV_{2j-1}$;
        \item For $i=1:2j$\\
        $H_{i,2j-1}=(V_i(p_i,:))^{-1}\V(p_i,:)$;\\
        $\V=\V-V_iH_{i,2j-1}$;\\
        EndFor
        \item $[V_{2j+1},H_{2j+1,2j-1}]=\lu(\V)$;
        \item $[\sim,p_{2j+1}]=\max{(V_{2j+1})}$;
        \item $\V=A^{-1}V_{2j}$;
        \item For $i=1:2j+1$\\
        $H_{i,2j}=(V_i(p_i,:))^{-1}\V(p_i,:)$;\\
        $\V=\V-V_iH_{i,2j}$;\\
        EndFor
        \item $[V_{2j+2},H_{2j+2,2j}]=\lu(\V)$;
        \item $[\sim,p_{2j+2}]=\max{(V_{2j+2})}$;
        \item EndFor
    \end{enumerate}
    \end{enumerate}
\end{algorithm}

The extended block Hessenberg process with partial pivoting (EBHA) is summarizing in Algorithm \ref{algorithm:EBHA}. We notice that the vectors $\{p_1,p_2,\ldots,p_{2m+2}\}$ defined in \eqref{blockvectorsHessenberg} relations can be computed using the \textbf{$\max$} MATLAB function as shown in Algorithm \ref{algorithm:EBHA} (Lines (2),(3.d) and (3.h)). The systems of equations with the matrix $A$ in Algorithm \ref{algorithm:EBHA} (Lines $(1)$ and $(3.e)$) are solved by LU factorization and by using the backslash operator of Matlab, See the EBHA code in Appendix section. Now, we compute the operation requirements for the EBHA in Algorithm \ref{algorithm:EBHA}. We recall the elementary flops:
\begin{itemize}
    \item $AV_{2j-1}$ requires $C_1^{Nz,p}=pNz$, where $Nz$ is the number of nonzero elements of matrix $A$.
    \item $A^{-1}V_{2j}$ requires $C_2^{n,p}=n(n+1)p.$
    \item The LU-factorization of some matrix of size $n\times p$ requires $C_3^{n,p}=p^2(n-p/3).$
    \item The computation of $H_{i,j}$ requires $C_4^p=5p^3/3+p^2$ (we assume that it is computed by means LU-factorization of $V_i(p_i,:)$).
    \item The computation of $V_i H_{i,j}$ requires $C_5^{n,p}=np^2.$
\end{itemize}

Then, the EBHA Algorithm requires \begin{align*}
Fl(EBHA)&=C_3^{n,2p}+\sum\limits_{j=1}^m[C_1^{Nz,p}+\sum\limits_{i=1}^{2j}(C^p_4+C^{n,p}_5)+C^{n,p}_3+C_2^{n,p}+\sum\limits_{i=1}^{2j+1}(C^p_4+C^{n,p}_5)+C^{n,p}_3],\\
&=mpNz+n(n+1)pm+[4p^2(n-2p/3)+mp^2(n-p/3)]+p^2m(3n+5p+3)(2m+3)/3.
\end{align*}

We next discuss some useful properties of the  extended block Hessenberg process.  Here and below we will tacitly assume that the number of steps of the  extended  block Hessenberg process is small enough to avoid breakdown. This is the generic situation; breakdown is very rare. Then Algorithm \ref{algorithm:EBHA} determines a $(2m+1)p\times (2mp)$ upper block Hessenberg matrix $\Hm_{2m}=[H_{i,j}]$ with $H_{i,j}\in\R^{p\times p}$.

Now, define the permutation matrix by $\mathbb{P}_{2m}=(\E_1,\ldots,\E_{2m})\in\R^{n\times 2mp}$ and set 
\begin{align}\label{PV}
    \mathbb{L}_{2m}=\mathbb{P}_{2m}^T\Vm_{2m}\in\R^{2mp\times 2mp},
\end{align}
according to \eqref{Galerkin}, the matrix $\mathbb{L}_{2m}$ is a unit lower triangular matrix. Let $\Vm_{2m}^L$  be the left inverse of $\mathbb{V}_{2m},$ defined by
\begin{align}\label{VL2m}
\Vm_{2m}^L=\mathbb{L}_{2m}^{-1}\mathbb{P}_{2m}^T\in\R^{2mp\times n}.
\end{align}

\noindent It is easy to see that $\Vm_{2m}^L\Vm_{2m}=I_{2mp},$ by using the equation \eqref{PV}. We now introduce the $2mp\times 2mp$ matrix given by
\begin{align}\label{TmDefinition}
\mathbb{T}_{2m}=[T_{i,j}]_{i,j=1}^{2m}=\mathbb{V}_{2m}^L A\mathbb{V}_{2m}\in\R^{2mp\times 2mp},
\end{align}
with $T_{i,j}=\Vm_{2m}^L(p(i-1)+1:ip,:)A\Vm_{2m}(:,p(j-1)+1:jp)\in\R^{p\times p}$, $i,j=1,\ldots,m$.

\noindent Define the following matrix as 
\[
\mathbb{V}_{2m+1}^L=\mathbb{L}_{2m+1}^{-1}\mathbb{P}^T_{2m+1}\in\R^{(2m+1)p\times n},
\]
where
\[
\mathbb{L}_{2m+1}=\mathbb{P}^T_{2m+1}\mathbb{V}_{2m+1}\in\R^{(2m+1)p\times (2m+1)p}\text{ and } \mathbb{P}_{2m+1}=[\widetilde{E}_1,\ldots,\widetilde{E}_{2m},\widetilde{E}_{2m+1}]\in\R^{n\times (2m+1)p}.
\]
\begin{proposition}
Assume that $m$ steps of Algorithm \ref{algorithm:EBHA} have been run and let $\widetilde{\mathbb{T}}_{2m}=\mathbb{V}_{2m+1}^LA\mathbb{V}_{2m}\in\R^{(2m+1)p\times 2mp}$, with $\mathbb{V}_{2m+1}=[V_1,\ldots,V_{2m},V_{2m+1}]\in\mathbb{R}^{n\times (2m+1)p}$, then we have the following relation
\begin{equation}\label{decompositionAV}
\begin{array}{rl}
A\mathbb{V}_{2m}&=\mathbb{V}_{2m+1}\widetilde{\mathbb{T}}_{2m}\\
&=\mathbb{V}_{2m}\mathbb{T}_{2m}+V_{2m+1}\tau_m E^T_m,
\end{array}
\end{equation}
where the matrix $E_m=[e_{2p(m-1)+1},\ldots,e_{2mp}]\in\R^{2mp\times 2p}$ is made up of the last  $2p$ columns of the identity matrix $I_{2mp}$ with $e_i$ is the $i$-th vector of the canonical basis of $\R^{2mp}$, $\tau_m=\begin{bmatrix}T_{2m+1,2m-1},~
T_{2m+1,2m}\end{bmatrix}$.
\end{proposition}

\begin{proof}
According to the recursion formulas \eqref{blockvectorsHessenberg}, we get for $j=1,\ldots,m$
\begin{align}
    AV_{2j-1} \text{ and } AV_{2j}\in{\rm range}\{V_1,\ldots,V_{2j+1}\}.
\end{align}
Hence,
\[
A\mathbb{V}_{2m}\in{\rm range}\{V_1,\ldots,V_{2m+1}\}.
\]
Then there exists a matrix $T\in\mathbb{R}^{(2m+1)p\times 2mp}$ such that
\begin{align}\label{VTT}
A\mathbb{V}_{2m}=\mathbb{V}_{2m+1}T.
\end{align}

Multiplying the equation \eqref{VTT} by $\mathbb{V}_{2m+1}^L$ from the left gives
\[
\mathbb{V}_{2m+1}^LA\mathbb{V}_{2m}=T.
\]
It follows that $\widetilde{\mathbb{T}}_{2m+1}=T.$ Since $\widetilde{\mathbb{T}}_{2m+1}$ is an upper block Hessenberg matrix with $2p\times 2p$ blocks, then $\mathbb{V}_{2m+1}\widetilde{\mathbb{T}}_{2m+1}$ can be decomposed as follows
\[
\mathbb{V}_{2m+1}\widetilde{\mathbb{T}}_{2m+1}=A\mathbb{V}_{2m}=\mathbb{V}_{2m}\mathbb{T}_{2m}+V_{2m+1}\tau_m E^T_m.
\]
Which completes the proof.
\end{proof}

 \noindent As the extended block Arnoldi process \cite{HJ2}, the entries of $\mathbb{T}_{2m}$ and $\widetilde{\mathbb{T}}_{2m}$ can be expressed in terms of recursion coefficients for the extended block Hessenberg process as shown below. This makes them easy to compute.

\begin{proposition}\label{proposition:ComputeTm}
Let the matrices $H_{i,j}$ and $\Gamma_{i,j}$ as defined in \eqref{blockvectorsHessenberg}, \eqref{ComputeV1V2},\eqref{V2} and \eqref{gamma12}; respectively. The matrix $\mathbb{T}_{2m}=[T_{i,j}]$ in \eqref{TmDefinition} is block upper Hessenberg with $2p\times 2p$ blocks with the nontrivial entries,
\begin{align}\label{TT1}
T_{:,2j-1}=H_{:,2j-1},\quad j=1,\ldots,m,
\end{align}
\begin{align}\label{TT2}
T_{:,2}=\widehat{E}_1\Gamma_{11}\Gamma_{2,2}^{-1}-H_{:,1}\Gamma_{1,2}\Gamma_{22}^{-1},
\end{align}
For $j=1,\ldots,m-1$,
\begin{align}\label{TT3}
T_{:,2j+2}=\Eh_{2j}H^{-1}_{2j+2,2j}-\s_{i=1}^{2j+1}H_{:,i}H_{i,2j}H_{2j+2,2j}^{-1},
\end{align}
where $T_{:,i}=\mathbb{T}_{2m}\widehat{E}_i\in\R^{2mp\times p},$ $H_{:,i}=\mathbb{H}_{2m}\widehat{E}_i\in\R^{2mp\times p},$  and  $\Eh_i=[\widehat{e}_i\otimes I_p]\in\R^{2mp\times p}$, for $i=1,\ldots,2m,$ where $\otimes$ denotes the Kronecker product and $\widehat{e}_i$ is the $i$-th vector of the canonical basis of $\R^{2m}$.
\end{proposition}

\begin{proof}
We have $T_{:,2j-1}=\mathbb{V}_{2m}^LAV_{2j-1}.$ Therefore, \eqref{TT1} follows from the expression of $H_{:,2j-1}$ in \eqref{HHH}. Using \eqref{V2}, we obtain
\[
A^{-1}V_1\Gamma_{1,1}=V_1\Gamma_{1,2}+V_2\Gamma_{2,2}.
\]
Multiplying this equation by $A$ from the left yields 
\[
V_1\Gamma_{1,1}=AV_1\Gamma_{1,2}+AV_2\Gamma_{2,2}.
\]
Then the vector $AV_2$ is written as follows
\[
AV_2=V_1\Gamma_{1,1}\Gamma^{-1}_{2,2}-AV_1\Gamma_{1,2}\Gamma^{-1}_{2,2}.
\]
The relation \eqref{TT2} is obtained by multiplying the expression $AV_2$ by $\mathbb{V}_{2m}^L$ from the left, i.e.,
\begin{align*}
    T_{:,2}&=\mathbb{V}_{2m}^LAV_2=\mathbb{V}_{2m}^LV_1\Gamma_{1,1}\Gamma_{2,2}^{-1}-\mathbb{V}_{2m}^LAV_1\Gamma_{1,2}\Gamma_{2,2}^{-1}\\
    &=\widehat{E}_1\Gamma_{1,1}\Gamma_{2,2}^{-1}-H_{:,1}\Gamma_{1,2}\Gamma_{2,2}^{-1}.
\end{align*}

\noindent The formula \eqref{TT3} is obtained from the expression of $AV_{2j+2},$ for $j=1,\ldots,m-1.$ Thus, multiplying the second equality in  \eqref{blockvectorsHessenberg} by $A$ from the left gives 
\[
AV_{2j+2}H_{2j+2,2j}=V_{2j}-\sum\limits_{i=1}^{2j+1}AV_iH_{i,2j}.
\]
Then,
\[
AV_{2j+2}=V_{2j}H_{2j+2,2j}^{-1}-\sum\limits_{i=1}^{2j+1}AV_iH_{i,2j}H_{2j+2,2j}^{-1}.
\]
The expression \eqref{TT3} is easily obtained  by multiplying the expression $AV_{2j+2}$ by $\mathbb{V}_{2m}^L$  from the left and using the fact that $\mathbb{V}_{2m}^LV_{2j}=\widehat{E}_{2j}$ and $H_{:,i}=\mathbb{V}_{2m}^LAV_i$. This concludes the proof of the proposition.
\end{proof}

Next, we show some auxiliary results on properties of the projection matrix $\mathbb{T}_{2m}$ defined in \eqref{TmDefinition} and the inverse projection matrix
\begin{align}\label{SmDefinition}
\mathbb{S}_{2m}=[S_{i,j}]=\mathbb{V}_{2m}^L A^{-1}\mathbb{V}_{2m}\in\R^{2mp\times 2mp},
\end{align}
with $S_{i,j}=\Vm_{2m}^L(p(i-1)+1:ip,:)A^{-1}\Vm_{2m}(:,p(j-1)+1:jp)$, $i,j=1,2,\ldots,m$. The matrix $\mathbb{S}_{2m}$ satisfies the following decomposition
\begin{align}\label{decompositionAinvV}
A^{-1}\mathbb{V}_{2m}=\mathbb{V}_{2m}\mathbb{S}_{2m}+[V_{2m+1},V_{2m+2}]\begin{bmatrix}0 & S_{2m+1,2m}\\
0&  S_{2m+2,2m}\end{bmatrix}E^T_m,
\end{align}

The following result relates positive powers of $\mathbb{S}_{2m}$ to negative powers of $\mathbb{T}_{2m}$.

\begin{lemma}\label{lemma:positivenegativePowersST}
Let $\mathbb{T}_{2m}$ and $\mathbb{S}_{2m}$ be given by \eqref{TmDefinition} and \eqref{SmDefinition}; respectively, and let $E_1=[e_{1},\ldots,e_{p}]\in\R^{2mp\times p}$. Then
\begin{align}\label{positivenegativePowersST}
    \mathbb{S}_{2m}^j
E_1=\mathbb{T}_{2m}^{-j}E_1,\quad \rm{for} \quad j=1,\ldots,m.
\end{align}
\end{lemma}

\begin{proof}
Using induction, we begin by showing that
\[
\Tm^j\Sm^j=\Tm^{j-1}\Sm^{j-1}-\Tm^{j-1}(\Vm_{2m}^L)A[V_{2m+1},V_{2m+2}])\widetilde{S}_{m+1,m}E^T_m\Sm^{j-1},\quad j=1,\ldots,m,
\]
with
\[
\widetilde{S}_{m+1,m}=\begin{bmatrix}
0 & S_{2m+1,2m}\\
0 & S_{2m+2,2m}
\end{bmatrix},
\]
Using the decomposition \eqref{decompositionAV} and \eqref{decompositionAinvV}, we obtain
\[
I_{2mp}=\Tm\Sm+(\Vm_{2m}^LA[V_{2m+1},V_{2m+2}])\widetilde{S}_{m+1,m}E^T_m,
\]
Let $j=2,3,\ldots,m$ and assume that \[
\Tm^k\Sm^k=\Tm^{k-1}\Sm^{k-1}-\Tm^{k-1}(\Vm_{2m}^L)A[V_{2m+1},V_{2m+2}])\widetilde{S}_{m+1,m}E^T_m\Sm^{k-1},\quad k=1,\ldots,j-1,
\]
by induction, we have
\begin{align*}
 \Tm^j\Sm^j&=\Tm\Tm^{j-1}\Sm^{j-1}\Sm\\
 &=\Tm^{j-1}\Sm^{j-1}-\Tm^{j-1}(\Vm_{2m}^LA[V_{2m+1},V_{2m+2}])\widetilde{S}_{m+1,m}E^T_m\Sm^{j-1},
\end{align*}
multiplying this equation by $E_1$ from the right gives
\[
\Tm^j\Sm^jE_1=\Tm^{j-1}\Sm^{j-1}E_1-\Tm^{j-1}(\Vm_{2m}^LA[V_{2m+1},V_{2m+2}])\widetilde{S}_{m+1,m}E^T_m\Sm^{j-1}E_1,
\]
exploiting the structure of the matrix $\Sm$, we have $E^T_m\Sm^{j-1}E_1=0,$ for $j=0,1,\ldots,m-1$. Then
\[
\Tm^j\Sm^jE_1=\Tm^{j-1}\Sm^{j-1}E_1=E_1,\quad j=1,\ldots,m.
\]
This completes the proof.

\end{proof}

\section{Application to the approximation of matrix functions }

In this section, we present the approximations of $\mathbb{I}(f)$ in \eqref{If} using the EBHA method. As in \cite{DK,KS,Simoncini,Simoncini2}, the approximation of $f(A)V$ is given by
\begin{align}\label{Ifapp}
    \mathbb{I}_{2m}(f):=\Vm_{2m}f(\mathbb{T}_{2m})E_1\Gamma_{1,1}.
\end{align}

The $n\times 2mp$ matrix $\Vm_{2m}=[V_1,\ldots,V_{2m}]$ is the matrix corresponding to the trapezoidal basis for $\mathbb{K}^e_m(A,V)$ constructed by applying $m$ steps of Algorithm \ref{algorithm:EBHA} to the pair $(A,V)$. $\mathbb{T}_{2m}$ is the projected matrix defined by \eqref{TmDefinition}, $E_1\in\R^{2mp\times p}$ is the first $p$ columns of the identity matrix $I_{2mp}$ and $\Gamma_{1,1}$ is a $p\times p$ matrix given by \eqref{ComputeV1V2}.

\begin{lemma}\label{lemma:exactitude}
Let $\mathbb{T}_{2m}$ and $\mathbb{S}_{2m}$ be defined by \eqref{TmDefinition} and \eqref{SmDefinition}; respectively, and let $\Vm_{2m}=[V_1,V_2,\ldots,V_{2m}]$ be the matrix computed in \eqref{ComputeV1V2} and \eqref{blockvectorsHessenberg}. Then \begin{eqnarray}
    A^jV_1&=\Vm_{2m}\mathbb{T}_{2m}^jE_1\quad &j=0,1,\ldots,m-1,\label{AjV1T}\\
    A^{-j}V_1&=\Vm_{2m}\mathbb{S}_{2m}^jE_1\quad &j=0,1,\ldots,m,\label{AinvjV1S}\\
    A^{-j}V_1&=\Vm_{2m}\mathbb{T}_{2m}^{-j}E_1\quad &j=0,1,\ldots,m.\label{AinvjV1T}
\end{eqnarray}
\end{lemma}

\begin{proof}
We obtain from \eqref{decompositionAV} that
\[
A\mathbb{V}_{2m}=\mathbb{V}_{2m}\Tm+V_{2m+1}
\tau_m E^T_m.
\]
Multiplying this equation by $E_1$ from the right gives
\[
AV_{1}=A\mathbb{V}_{2m}E_{1}=\mathbb{V}_{2m}\Tm E_1
+V_{2m+1}\tau_m E_m^T E_{1}=\mathbb{V}_{2m}\Tm E_1.
\]
Let $j=2,3,\ldots,m-1$, and assume that
\[
A^{k}V_{1}=\mathbb{V}_{2m}\Tm^{k}E_{1},\quad k=0,1,\ldots,j-1.
\]
We will show the identity
\[
A^{j}V_{1}=\mathbb{V}_{2m}\Tm^{j}E_{1},
\]
by induction. We have
\[
A^{j}V_{1}=A\cdot A^{j-1}V_{1}=A\mathbb{V}_{2m}\Tm^{j-1}E_{1}.
\]
Using the decomposition \eqref{decompositionAV}, we obtain
\begin{align*}
A^{j}V_{1}&=\begin{bmatrix}\mathbb{V}_{2m}\Tm+V_{2m+1}
\begin{pmatrix}\tau_m
E^{T}_m\end{pmatrix}\end{bmatrix}\bigg(T^{j-1}_{2m}E_{1}\bigg)\\
&=\mathbb{V}_{2m}\Tm^{j}E_{1}+V_{2m+1}
\begin{pmatrix}\tau_mE^T_{m}
T^{j-1}_{2m}E_{1}\end{pmatrix}.
\end{align*}
Exploiting the structure of the matrix $\Tm$, we have $E^T_m\Tm^{j}E_1=0$, for $j=0,1,\ldots,m-2$. Then
\[
A^{j}V_{1}=\mathbb{V}_{2m}\Tm^{j}E_{1}.
\]
According to \eqref{decompositionAinvV} and by using the same techniques as
above, we find \eqref{AinvjV1S}. Finally, \eqref{AinvjV1T} follows from \eqref{AinvjV1S} and Lemma \ref{lemma:positivenegativePowersST}.
\end{proof}

According to results of this Lemma, we observe that the approximation \eqref{Ifapp} is exact for Laurent polynomials of positive degree at most $m-1$, and negative degree at most $m$.

\subsection{ Approximations for \boldmath{$\exp(A)V$}}

In this subsection, we consider, the approximation of $\exp(A)V$, which is given by   $\Vm_{2m}\exp(\mathbb{T}_{2m})E_1\Gamma_{1,1}$. In the following proposition, we give an upper error bound associated to this approximation. This result holds when the matrix $A$ satisfies the following assumption $x^TAx\leq 0$, $\forall x\in\R^n$.

\begin{proposition}
Let $\Vm_{2m}$ be the matrix computed in \eqref{ComputeV1V2} and \eqref{blockvectorsHessenberg} and let $\mathbb{T}_{2m}$ be the projected matrix defined by \eqref{TmDefinition}. Assume that the matrix $A$ satisfies $x^TAx\leq 0$ $\forall x\in\R^n$.  Then the approximation error satisfies
\begin{align}
    \|\mathbb{I}(\exp)-\mathbb{I}_{2m}(\exp)\|_2\leq C_{2m}\dfrac{1-e^{\mu_2(A)}}{-\mu_2(A)},
\end{align}
where $C_{2m}:=\|V_{2m+1}\tau_mE^T_m\exp(\mathbb{T}_{2m})E_1\Gamma_{1,1}\|_2$,  $\mu_2(A):=\frac{1}{2}\lambda_{max}(A+A^T)\leq0,$ and  $\|\cdot\|_2$ corresponds to the spectral norm.
\end{proposition}

\begin{proof}
Consider $X(t)=\exp(tA)V$ and $X_{2m}(t)=\Vm_{2m}\exp(t\mathbb{T}_{2m})E_1\Gamma_{1,1}$ with $t>0$. Then it follows that the derivative of $X(t)$ can be written as
\[
X^{'}(t)=A\exp(tA)V=AX(t),\quad t>0.
\]
Using \eqref{decompositionAV}, we obtain
\[
X^{'}_{2m}(t)=AX_{2m}-V_{2m+1}\tau_m E^T_m\exp(t\mathbb{T}_{2m})E_1\Gamma_{1,1}.
\]
Define the approximate error as  $E_{2m}(t):=X(t)-X_{2m}(t)$, then the derivative of $E_{2m}(t)$ is computed as
\begin{align*}
    E^{'}_{2m}(t)&=X^{'}(t)-X^{'}_{2m}(t)\\
    &=AX(t)-AX_{2m}(t)+V_{2m+1}\tau_m E^T_m\exp(t\mathbb{T}_{2m})E_1\Gamma_{1,1}.
  \end{align*}
Then the approximate error $E_{2m}$ satisfies the following equation
\begin{equation}\label{Eqq}
    \begin{cases}
    E^{'}_{2m}(t)&=AE_{2m}(t)+V_{2m+1}\tau_m E^T_m\exp(t\mathbb{T}_{2m})E_1\Gamma_{1,1},\quad t>0\\
    E_{2m}(0)&=0.
    \end{cases}
\end{equation}
This equation is a particular case of the general differential Sylvester equation (see. e.g., \cite{AFIJ,HaJ} for more details). Then the error $E_{2m}(t)$ is written as
\[
E_{2m}(t)=\int^{t}_0 \exp((t-s)A)V_{2m+1}\tau_m E^T_m\exp(s\mathbb{T}_{2m})E_1\Gamma_{1,1}ds,
\]
and for $t=1,$ we have
\[
E_{2m}(1)=\int^{1}_0 \exp((1-s)A)V_{2m+1}\tau_m E^T_m\exp(s\mathbb{T}_{2m})E_1\Gamma_{1,1}ds,
\]
\begin{align*}
\|\mathbb{I}(\exp)-\mathbb{I}_{2m}(\exp)\|_2=\|E_{2m}(1)\|_2&\leq \int^1_0 \|\exp((1-s)A)V_{2m+1}\tau_m E^T_m\exp(s\mathbb{T}_{2m})E_1\Gamma_{1,1}\|_2 ds\\
&\leq \int^1_0 \|\exp((1-s)A)\|_2 \|V_{2m+1}\tau_m E^T_m\exp(s\mathbb{T}_{2m})E_1\Gamma_{1,1}\|_2 ds\\
&\leq \max\limits_{\lambda\in[0,1]}\|V_{2m+1}\tau_m E^T_m\exp(\lambda\mathbb{T}_{2m})E_1\Gamma_{1,1}\|_2\int^1_0 \|\exp((1-s)A)\|_2 ds.
\end{align*}
Since $x^TAx\leq 0$ $\forall x\in\R^n$, then the use of the logarithmic norm yields (see. e.g., \cite[Section I.2.3]{HV}),  
\[
\|\exp(tA)\|_2\leq e^{t\mu_2(A)}, \quad \forall t>0.
\]
 Hence,
\[
\|\exp((1-s)A)\|_2\leq e^{(1-s)\mu_2(A)},\quad \text{for all } 0<s<1.
\]

\begin{align*}
\|\mathbb{I}(\exp)-\mathbb{I}_{2m}(\exp)\|_2&\leq \max\limits_{\lambda\in[0,1]}\|V_{2m+1}\tau_m E^T_m\exp(\lambda\mathbb{T}_{2m})E_1\Gamma_{1,1}\|_2\int^1_0 e^{(1-s)\mu_2(A)} ds\\
&\leq C_{2m} \dfrac{1-e^{\mu_2(A)}}{-\mu_2(A)}.
\end{align*}

\end{proof}

\noindent Algorithm \ref{algorithm:MFEBH} describes how approximations of $f(A)V$ are computed by the extended block Hessenberg method.

\begin{algorithm}
 \caption{Approximation of $f(A)V$ by the extended block Hessenberg method (MF-EBH)}\label{algorithm:MFEBH}
 \textbf{Inputs:} Matrix $A$, initial block $V\in\R^{n\times p}$, an integer $m$ and a function $f$.
 \begin{enumerate}
    \item $[\Vm_2,\Gamma]=\lu([V,A^{-1}V])$; where $\Vm_2=[V_1,V_2]$,
    \item $[\sim,p_1]=\max{(V_1)}$; and $[\sim,p_2]=\max{(V_2)}$;
    \item For $j=1:m$
    \begin{enumerate}
        \item $\V=AV_{2j-1}$;
        \item For $i=1:2j$\\
        $H_{i,2j-1}=(V_i(p_i,:))^{-1}\V(p_i,:)$;\\
        $\V=\V-V_iH_{i,2j-1}$;\\
        EndFor
        \item $[V_{2j+1},H_{2j+1,2j-1}]=\lu(\V)$;
        \item $[\sim,p_{2j+1}]=\max{(V_{2j+1})}$;
        \item $\V=A^{-1}V_{2j}$;
        \item For $i=1:2j+1$\\
        $H_{i,2j}=(V_i(p_i,:))^{-1}\V(p_i,:)$;\\
        $\V=\V-V_iH_{i,2j}$;\\
        EndFor
        \item $[V_{2j+2},H_{2j+2,2j}]=\lu(\V)$;
        \item $[\sim,p_{2j+2}]=\max{(V_{2j+2})}$; $\mathbb{V}_{2m+2}=[\mathbb{V}_{2m},V_{2m+1},V_{2m+2}].$
        \item EndFor
    \end{enumerate}
    \item Compute $\mathbb{T}_{2m}$ by using Proposition \ref{proposition:ComputeTm}.
    \item Set $\Gamma_{1,1}\in\R^{p\times p}$ be the first $p$ columns and $p$ rows of the matrix $\Gamma$.
    \item $\mathbb{I}_{app}(f)=\mathbb{V}_{2m}f(\mathbb{T}_{2m})E_1\Gamma_{1,1}$.
    \end{enumerate}
 \textbf{Output:} Approximation $\mathbb{I}_{app}(f)$ of the matrix function $f(A)V$.
 \end{algorithm}

\section{Shifted block linear systems}

We consider the solution of the parameterized nonsingular linear systems with multiple right hand sides
\begin{equation}\label{SL}
    (A+\sigma I_n)X^{\sigma}=C,\quad \sigma\in\Sigma.
\end{equation}
 $\Sigma$ is the set of the shifts. Then the approximate solutions $X_{2m}^{\sigma}\in\mathbb{R}^{n\times p}$ generated by the  extended block Hessenberg method to the pair $(A,R_0^{\sigma})$ are obtained as follows
\begin{equation}\label{XZ}
    X_{2m}^{\sigma}:=X_0^{\sigma}+Z_{2m}^{\sigma};\quad Z_{2m}^{\sigma}\in \mathbb{K}^{e}_{m}(A,R_0^{\sigma}),
\end{equation}
where $R^{\sigma}_0:=C-(A-\sigma I_n)X_0^{\sigma}$ are the residual block vectors associated to the initial guess $X_0^{\sigma}$. Since $Z_{2m}(\sigma)\in \mathbb{K}^e_m(A,R_0^{\sigma}),$ then,
\begin{equation}\label{ZZ}
Z_{2m}^{\sigma}=\mathbb{V}_{2m}Y_{2m}^{\sigma},\quad Y_{2m}^{\sigma}\in\mathbb{R}^{2mp}.
\end{equation}
$Y_{2m}^{\sigma}$ is determined such that the new residual $R_{2m}^{\sigma}=C-(A+\sigma I_n)X_{2m}^{\sigma}$ associated  to $X_{2m}^{\sigma}$ is orthogonal to $\mathbb{K}^{e}_{m}(A,R_0)$. This yields
\begin{equation}\label{OC}
     \mathbb{V}_{2m}^L R_{2m}^{\sigma}=0,
\end{equation}
where $\mathbb{V}^L_{2m}$ is the left inverse of the matrix $\mathbb{V}_{2m}$ defined by \eqref{VL2m}. According to \eqref{decompositionAV} equation, we obtain
\[
(A+\sigma I_n)\mathbb{V}_{2m}=\mathbb{V}_{2m}(\mathbb{T}_{2m}+\sigma I_{2mp})+V_{2m+1}\tau_m E^T_m,
\]
and $R^{\sigma}_{0}=V_1\Gamma_{1,1}$. Using this equation and \eqref{OC} relations, the reduced linear system can be written as
\begin{equation}\label{RLS}
    (\mathbb{T}_{2m}+\sigma I_{2mp})Y_{2m}^{\sigma}=E_1\beta_0^{\sigma},\quad \beta_0^{\sigma}=\Gamma_{1,1}.
\end{equation}
Using \eqref{XZ} and \eqref{ZZ} equations, we get the approximate solution
\begin{equation}\label{X2m}
X_{2m}^{\sigma}=X_0^{\sigma}+\mathbb{V}_{2m}(\mathbb{T}_{2m}+\sigma I_{2mp})^{-1}E_1\Gamma_{1,1}.
\end{equation}

 \noindent The following result on  the residual $R_{2m}^{\sigma}$ allows us to stop the iterations without having to compute matrix products with the large matrix $A$ .
 \begin{theorem}
 Let $Y_{2m}^{\sigma}$ be the exact solution of the reduced linear system \eqref{RLS} and let $X_{2m}^{\sigma}$ be the approximate solution of the linear system \eqref{SL} after $m$ steps of the extended block Hessenberg method applied to the pair $(A,R_0)$. Then the residual $R_{2m}^{\sigma}$ satisfies
 \begin{align}\label{ResSL}
     R_{2m}^{\sigma}=-V_{2m+1}\tau_mE^T_mY_{2m}^{\sigma},
 \end{align}
 where $\tau_m,$ $E_m$ are the matrices as defined in \eqref{decompositionAV}.
 \end{theorem}

 \begin{proof}
 We use \eqref{RLS} equation in this proof, and the fact that $R^{\sigma}_{0}-V_1\Gamma_{1,1}$, then
 \begin{align*}
     R_{2m}^{\sigma}&=C-(A+\sigma I_n)X_{2m}^{\sigma}=C-(A+\sigma I_n)(X_{0}^{\sigma}+\mathbb{V}_{2m}Y_{2m}^{\sigma})\\
     &=R^{\sigma}_0-(A+\sigma I_n)\mathbb{V}_{2m}Y_{2m}^{\sigma}\\
     &=R^{\sigma}_{0}-\mathbb{V}_{2m}(\mathbb{T}_{2m}+\sigma I_{2mp})Y_{2m}^{\sigma}-V_{2m+1}\tau_mE^T_mY_{2m}^{\sigma}\\
     &=R^{\sigma}_{0}-\mathbb{V}_{2m}E_1\Gamma_{1,1}-V_{2m+1}\tau_mE^T_mY_{2m}^{\sigma}\\
     &=R^{\sigma}_{0}-V_1\Gamma_{1,1}-V_{2m+1}\tau_mE^T_mY_{2m}^{\sigma}\\
      &=-V_{2m+1}\tau_mE^T_mY_{2m}^{\sigma}.
     \end{align*}
     Which completes the proof.
 \end{proof}

\noindent Since all basis vectors $\{V_1,\ldots,V_{2m}\}$ need to be stored, a maximum subspace dimension is usually allowed, but when accuracy of the approximation \eqref{X2m} is not satisfactory at this maximum dimension, the procedure should to be restarted with the current approximate solution as a starting guess, and the new space is generated with the current residual as a starting vector. According to \eqref{ResSL} equation, we observe that   $R_{2m}^{\sigma}\in{\rm range}\{V_{2m+1}\}$. Then, it is possible to restart the algorithm for every some fixed $m$ steps to solve the linear system \eqref{SL} with
\[
V_1=V_{2m+1}, \quad \text{and}\quad
\beta_0^{\sigma}=V_{2m+1}^LR^{\sigma}_{2m}=-\tau_m E^T_mY^{\sigma}_{2m}.
\]
We refer to  \cite{SCS,Simoncini2} for more details on restarting procedure for shifted linear systems.

\begin{algorithm}
\caption{Restarted shifted linear system algorithm using the extended block Hessenberg algorithm (restarted-EBH)}\label{algorithm:Restarted-EBH}
\textbf{Inputs:} Matrix $A$, block vector $C$,  the set of the shifts $\Sigma$,  a desired tolerance $\epsilon$ and an integer $m$.
\begin{enumerate}
\item $[V_1,\beta_0^{\sigma}]=lu(C)$ and set $\Sigma_c=\emptyset$, $X_{2m}^{\sigma}=O_{n\times p}$.
\item While $\Sigma\backslash\Sigma_c\neq \emptyset$
\begin{enumerate}
\item Compute $\mathbb{V}_{2m}$ and $\mathbb{T}_{2m}$ using $m$ steps of EBHA applied to the pair $(A,V_1)$ in Algorithm \ref{algorithm:EBHA}.
\item Solve the reduced shifted linear system  $(\mathbb{T}_{2m}+\sigma I_{2mp})Y^{\sigma}_{2m}=E_1\beta_0^{\sigma},$ for $\sigma\in\Sigma\backslash\Sigma_c$.
\item Compute $\|R_{2m}^{\sigma}\|_F=\|V_{2m+1}\tau_mE^T_mY_{2m}^{\sigma}\|_F$, for $\sigma\in\Sigma\backslash\Sigma_c$.
\item Compute $X_{2m}^{\sigma}=X_{2m}^{\sigma}+\mathbb{V}_{2m}Y_{2m}^{\sigma}$, for $\sigma\in\Sigma\backslash \Sigma_c$.
\item Select the new $\sigma\in\Sigma\backslash\Sigma_c$ such that $\|R_{2m}(\sigma)\|_F<\epsilon$. Update set $\Sigma_c$ of converged shifted linear systems.
\item Set $V_1=V_{2m+1}$ and $\beta_0^{\sigma}=-\tau_mE^T_mY_{2m}^{\sigma}$, for $\sigma\in\Sigma\backslash \Sigma_c.$
\item endwhile
\end{enumerate}
\end{enumerate}
\textbf{Output:} Approximation $X_{2m}^{\sigma}$ of the shifted linear systems \eqref{SL}, $\forall \sigma\in \Sigma$.
\end{algorithm}
\section{Numerical experiments}\label{section:ne}
In this section, we illustrate the performance of the extended block Hessenberg  (EBH) method when applied to reduce the order of large scale dynamical systems. All experiments were carried out in MATLAB R2015a on a computer with an Intel Core i-3 processor and $3.89$ GB of RAM. The computations were done with about $15$ significant decimal digits. In the selected examples, the proposed method is compared with the extended block Arnoldi (EBA) method \cite{AHJ,HJ,Heyouni}. In all numerical examples, the execution time performed is taken from the average of $10$ multiple runs. This is is important to make the timing more trustworthy. 

\subsection{Examples for the approximation of matrix functions \boldmath{$f(A)V$}}
The examples of this subsection  compare the performance of the extended block Hessenberg  (MF-EBH) Algorithm \ref{algorithm:MFEBH}, with the performance of the extended block Arnoldi algorithm (MF-EBA) when applied to the approximation of $f(A)V$. The matrix $A\in\R^{n\times n}$, with $n=5000.$ The initial block vector $V\in\R^{n\times p}$ is generated randomly with uniformly distributed entries in the interval $[0,1]$ and the block size $p$ is $5$. In the tables \ref{table:exampleA1},  and \ref{table:exampleA2}, we display the relative errors $\|\mathbb{I}(f)-\mathbb{I}_{2m}(f)\|/\|\mathbb{I}(f)\|$, and the required CPU time for MF-EBH and MF-EBA; respectively. We also report the ratio of execution times $t(p)/t(1)$, where $t(p)$ is the CPU time for the extended block Hessenberg method and $t(1)$ is the CPU time obtained when applying the extended Hessenberg method \cite{RT} for one right-hand side. This right-hand side $v\in\R^n$ is chosen to be the first column of $V$.  The number of iterations is set to $m=10$ and $m=15$. We used the {\sf funm} function in Matlab, to compute the exact solution $\mathbb{I}(f).$

\textbf{Example 1.}  Let $A=[a_{i,j}]$ be the symmetric positive definite Toeplitz matrix with entries $a_{i,j}=1/(1+|i-j|)$ \cite{JR}. The condition number of this matrix is $50.434$. The approximation errors and CPU times are listed in Table \ref{table:exampleA1} for several functions $f$.

\begin{table}[h]
    \centering
     \caption{Example $1$: Approximation of $f(A)V$ for several functions, $A\in\R^{n\times n}$, $V\in\R^{n\times p}$, $n=5000$ and $p=5$.}
    \begin{tabular}{c|ccc|cc}
    \hline
        \multirow{2}{2cm}{$f(x)$} & \multicolumn{2}{c}{MF-EBH} & \multicolumn{2}{c}{MF-EBA}  \\
        \cline{2-6}
         & Time $(s)$ & t($p$)/t($1$) & Ret. Err & Time $(s)$ & Ret. Err \\
         \hline
         $m=10$ & & & & \\
         \cline{1-1}
         $\exp(x)$ & $7.31$ & $2.32$ & $4.25\cdot 10^{-7}$ & $10.82$ & $1.03\cdot 10^{-7}$\\
         $\sqrt{x}$ & $7.57$ & $2.18$ &$9.78\cdot 10^{-10}$ & $10.45$ & $2.16\cdot 10^{-10}$\\
         $\exp(-\sqrt{x})$ & $10.51$ & $2.41$ & $2.01\cdot 10^{-8}$ & $12.31$ & $1.25\cdot 10^{-8}$\\
         $\log(x)$ & $10.35$ & $2.64$ & $2.94\cdot 10^{-9}$ & $12.56$ & $1.81\cdot 10^{-9}$\\
         $\exp(-x)/x$ & $10.53$ & $2.16$ &$4.29\cdot 10^{-8}$ & $12.27$ & $1.09\cdot 10^{-8}$\\
         \hline
         $m=15$ & & & & \\
         \cline{1-1}
         $\exp(x)$ & $11.03$ &  $2.39$ & $5.06\cdot 10^{-12}$ & $12.30$ & $1.10\cdot 10^{-12}$\\
         $\sqrt{x}$ & $9.32$ & $2.26$ & $3.64\cdot 10^{-14}$ & $12.57$ & $1.56\cdot 10^{-14}$\\
         $\exp(-\sqrt{x})$ & $12.78$ & $2.53$ & $7.94\cdot 10^{-13}$ & $15.35$ & $1.23\cdot 10^{-13}$\\
         $\log(x)$ & $13.98$ & $2.68$ & $1.14\cdot 10^{-13}$ & $15.63$ & $9.62\cdot 10^{-15}$\\
         $\exp(-x)/x$ & $13.34$ & $2.23$ & $2.49\cdot 10^{-13}$ & $15.97$ & $1.52\cdot 10^{-13}$
         \end{tabular}
    \label{table:exampleA1}
\end{table}

\textbf{Example 2.}  The matrix $A$ is a block diagonal with $2\times 2$ blocks of the form
\[
\begin{bmatrix}
a_i\, &\, c\\
-c \,&\, a_i
\end{bmatrix},
\]
in which $c=1/2$ and $a_i=(2i-1)/(n+1)$ for $i=1,\ldots,n/2$ \cite{Saad}. The condition number of this matrix is $3.62$. Results for several functions $f$ are reported in Table \ref{table:exampleA2}.

\noindent As can be seen in these two Tables, that MF-EBH have the best execution time than MF-EBA method for all functions.

\begin{table}[h]
    \centering
    \caption{Example $2$: Approximation of $f(A)V$ for several functions, $A\in\R^{n\times n}$, $V\in\R^{n\times p}$, $n=5000$ and $p=5$.}
    \begin{tabular}{c|ccc|cc}
    \hline
        \multirow{2}{2cm}{$f(x)$} & \multicolumn{2}{c}{MF-EBH} & \multicolumn{2}{c}{MF-EBA}  \\
        \cline{2-6}
         & Time $(s)$ & t($p$)/t($1$) &  Ret. Err & Time $(s)$ & Ret. Err \\
         \hline
         $m=10$ & & & & \\
         \cline{1-1}
         $\exp(x)$ & $0.27$ & $2.26$ & $8.06\cdot 10^{-11}$ & $0.57$ & $2.55\cdot 10^{-11}$\\
         $\sqrt{x}$ & $0.26$ & $2.08$ & $3.97\cdot 10^{-8}$ & $0.52$ & $1.42\cdot 10^{-8}$\\
         $\exp(-\sqrt{x})$ & $0.28$ & $2.29$ & $6.32\cdot 10^{-8}$ & $0.52$ & $2.26\cdot 10^{-8}$\\
         $\log(x)$ & $0.28$ & $2.46$ & $1.27\cdot 10^{-7}$ & $0.45$ & $9.54\cdot 10^{-9}$\\
         $\exp(-x)/x$ & $0.27$ & $2.08$ & $2.56\cdot 10^{-12}$ & $0.5$ & $1.41\cdot 10^{-12}$\\
         \hline
         $m=15$ & & & & \\
         \cline{1-1}
         $\exp(x)$ & $0.47$ & $2.34$ & $1.20\cdot 10^{-14}$ & $0.89$ & $4.47\cdot 10^{-15}$\\
         $\sqrt{x}$ & $0.47$ & $2.18$ & $1.19\cdot 10^{-11}$ & $0.89$ & $3.03\cdot 10^{-12}$\\
         $\exp(-\sqrt{x})$ & $0.55$ & $2.36$  & $1.91\cdot 10^{-11}$ & $0.88$ & $4.87\cdot 10^{-12}$\\
         $\log(x)$ & $0.52$ & $2.53$ & $3.85\cdot 10^{-11}$ & $0.82$ & $9.84\cdot 10^{-12}$\\
         $\exp(-x)/x$ & $0.52$ & $2.14$ & $1.88\cdot 10^{-14}$ & $0.79$ & $9.81\cdot 10^{-15}$
         \end{tabular}
    \label{table:exampleA2}
\end{table}
\textbf{Example 3.}  Let  $A=n^2\, tridiag(-1,2,-1),$ with $n=5000$. The condition number of this matrix is $1.25\cdot 10^7$. In this example, we compare the CPU time and the number of iterations needed by the MF-EBH and the MF-EBA methods so that the relative error reaches $2\cdot 10^{-9}.$ The number of iterations and timings are listed in Table \ref{table:exampleA3}. Timings show the proposed method to be faster than the MF-EBA method, even though the MF-EBA converges in less iterations. To illustrate how stable the MF-EBH method, the plots in Figure \ref{fig:Exemple3 } show the evolution of the approximation errors versus the number of iterations.  The figure demonstrates the stability of the MF-EBH method for this example.

\begin{table}[h]
    \centering
     \caption{Example $3$: CPU time, and number of iterations for Approximation of $f(A)V$ for several functions, $A\in\R^{n\times n}$, $V\in\R^{n\times p}$, $n=5000$ and $p=5$.}
    \begin{tabular}{c|cc|cc}
    \hline
        \multirow{2}{2cm}{$f(x)$} & \multicolumn{2}{c}{MF-EBH} & \multicolumn{2}{c}{MF-EBA}  \\
        \cline{2-5}
         & Time $(s)$ & Iterations  & Time $(s)$ & Iterations \\
         \hline
          $\sqrt{x}$ & $3.46$ & $34$ & $5.08$ & $33$\\
         $\exp(-\sqrt{x})$  & $0.78$ & $8$ & $1.56$ & $7$\\
         $\log(x)$  & $4.49$ & $35$ & $6.16$ & $33$\\
         \hline
         \end{tabular}
    \label{table:exampleA3}
\end{table}

\begin{figure}
    \centering
    \includegraphics[width=1\textwidth,height=7cm]{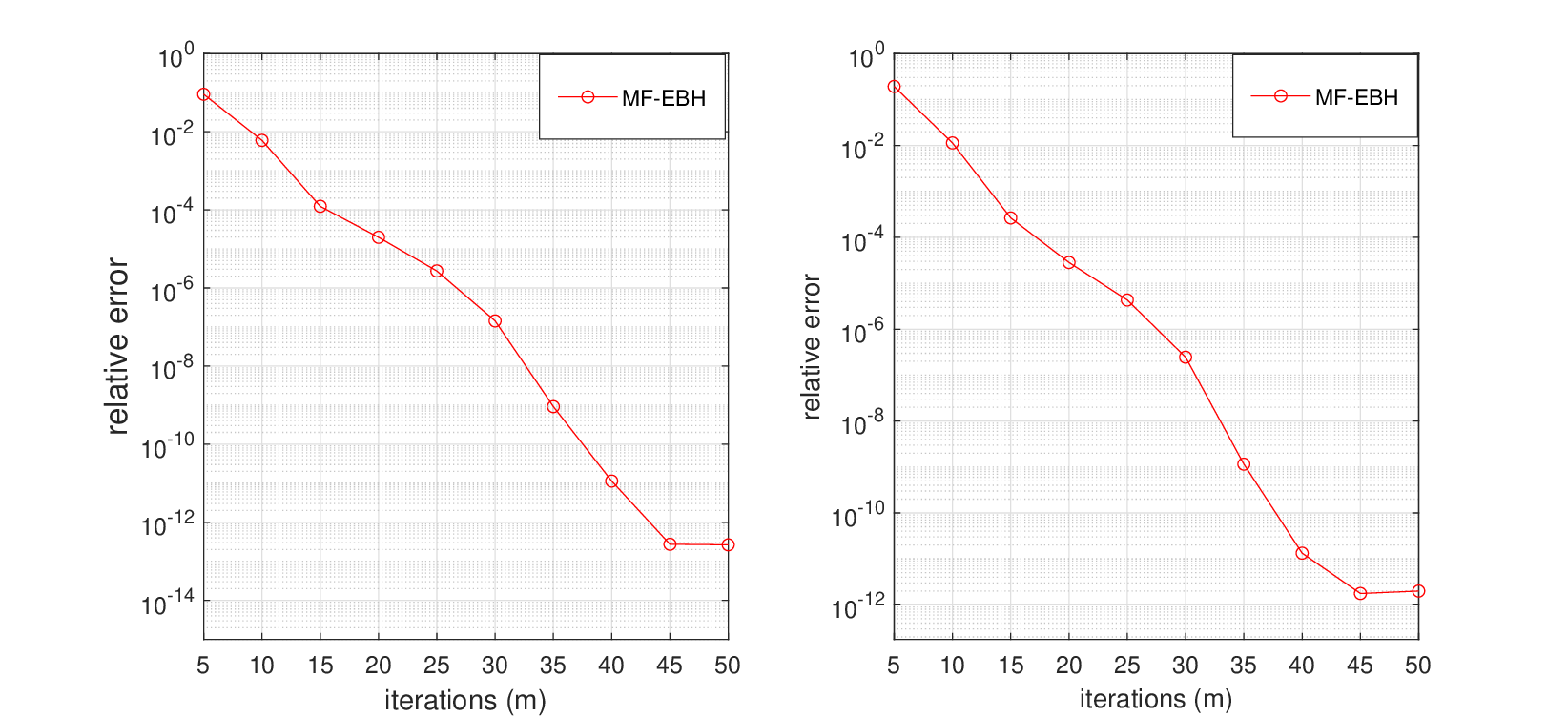}
    \caption{Evolution of the absolute error determined by the MF-EBH method when approximating $\sqrt{A}V$ (left plot) and $\log{(A)}V$ (right plot) }
    \label{fig:Exemple3 }
\end{figure}

\subsection{Examples of the shifted linear systems}
In this subsection, we present some results of solving shifted linear systems in \eqref{SL}. We compare the results obtained by the restarted-EBH Algorithm in Algorithm \ref{algorithm:Restarted-EBH}, the restarted extended block Arnoldi (restarted-EBA) in \cite{Simoncini2} and the Gaussian elimination with partial pivoting method (GE). It is a direct method based on the computation of the LU factorisation of the matrix $A+\sigma I_n$, for all $\sigma\in\Sigma.$ The right-hand side $C$ in \eqref{SL} is chosen  randomly with entries uniformly distributed on $[0,1]$, the block size is set to $p=5$. The shifts $\sigma$ are taken to be $500$ values uniformly distributed in the interval $[0,5]$. In all examples of this subsection, the stopping criteria is set to $R_{2m}\leq 2\cdot 10^{-8}$, where $R_{2m}:=\max\limits_{\sigma\in\{\sigma_1,\ldots,\sigma_{500}\}}\|R_{2m}^{\sigma}\|_F$  and the initial guess $X^{\sigma}_0$ is $0$. The number of iterations is $m=5,10,$ for the both algorithms.

\noindent\textbf{Example $4$.} In this example, we consider two nonsymmetric matrices which coming from the centered finite
difference discretization (CFDD) of the operators

\begin{equation}\label{OpSE}
\begin{array}{ll}
\mathcal{L}_1(u)&=-\Delta u+10u_x,\\
\mathcal{L}_2(u)&=-\Delta u+50(x+y)u_x+50(x+y)u_y.
\end{array}
\end{equation}
The operators $\mathcal{L}_1(u)$ and $\mathcal{L}_2(u)$ are given in \cite{Simoncini2} and \cite{Simoncini3}; respectively.

In Table \ref{table:exampleShiftedSystemsFDM}, we report results for restarted-EBH, and restarted-EBA in terms of the number of restarts ($\#$restarts), CPU time in seconds (time $(s)$) and the norm of the residual ($R_{2m}$). We also report the time obtained when applying the GE method (time $(s)$). We use different  values of the dimension $n$ (the size of the matrix $A$).

\begin{table}[h]
\centering
  \caption{Example $4$:  Shifted solvers for nonsymmetric matrices and different matrix dimensions for the operators given by \eqref{OpSE} }
 \label{table:exampleShiftedSystemsFDM}
    \begin{tabular}{cccccccccc}
    \hline
      Oper. & $n$ & Iter.$(m)$ & \multicolumn{3}{c}{restarted-EBH}   &  \multicolumn{3}{c}{restarted-EBA} & GE  \\
        & & & $\#$restarts & times$(s)$ & $R_{2m}$ & $\#$restarts & times$(s)$ & $R_{2m}$ & times$(s)$\\
         \hline
           $\mathcal{L}_1$ & $10000$ & $5$ & $2$ & $2.88$ & $1.99\cdot 10^{-8}$ & $2$ & $4.92$ & $1.97\cdot 10^{-8}$ & $10.78$\\
            &  & $10$ & $1$ & $2.21$ & $8.80\cdot 10^{-11}$ & $1$ & $3.26$ & $1.49\cdot 10^{-10}$ & \\
                & $22500$  & $5$ & $2$ & $3.62$ & $1.95\cdot 10^{-8}$ & $2$ & $5.96$ & $1.98\cdot 10^{-8}$ & $24.59$\\
            &  & $10$ & $1$ & $3.45$ & $3.02\cdot 10^{-10}$ & $1$ & $5.42$ & $5.79\cdot 10^{-10}$ & \\
             & $40000$  & $5$ & $2$ & $6.83$ & $1.93\cdot 10^{-8}$ & $2$ & $8.76$ & $1.94\cdot 10^{-8}$ & $34.27$\\
            &  & $10$ & $1$ & $7.91$ & $8.11\cdot 10^{-10}$ & $1$ & $9.48$ & $1.40\cdot 10^{-9}$ &\\
             & $625000$  & $5$ & $2$ & $10.12$ & $1.93\cdot 10^{-8}$ & $2$ & $12.81$ & $1.97\cdot 10^{-8}$ & $86.61$ \\
            &  & $10$ & $1$ & $12.85$ & $1.21\cdot 10^{-9}$ & $1$ & $16.24$ & $5.37\cdot 10^{-9}$ & \\
              $\mathcal{L}_2$ & $10000$ & $5$ & $1$ & $1.41$ & $7.90\cdot 10^{-9}$ & $1$ & $2.23$ & $1.39\cdot 10^{-8}$ & $8.48$ \\
            &  & $10$ & $1$ & $2.74$ & $6.85\cdot 10^{-11}$ & $1$ & $3.44$ & $6.20\cdot 10^{-11}$ & \\
                & $22500$  & $5$ & $2$ & $4.76$ & $1.97\cdot 10^{-8}$ & $1$ & $4.91$ & $1.91\cdot 10^{-8}$ & $18.52$\\
            &  & $10$ & $1$ & $3.87$ & $1.71\cdot 10^{-10}$ & $1$ & $6.39$ & $3.57\cdot 10^{-10}$ & \\
             & $40000$  & $5$ & $2$ & $5.65$ & $1.97\cdot 10^{-8}$ & $2$ & $8.03$ & $1.98\cdot 10^{-8}$ & $30.82$\\
            &  & $10$ & $1$ & $7.28$ & $4.45\cdot 10^{-10}$ & $1$ & $10.71$ & $8.75\cdot 10^{-10}$ & \\
             & $625000$  & $5$ & $2$ & $10.03$ & $1.99\cdot 10^{-8}$ & $2$ & $12.36$ & $1.99\cdot 10^{-8}$ & $76.25$\\
            &  & $10$ & $1$ & $11.41$ & $6.27\cdot 10^{-10}$ & $1$ & $16.68$ & $1.28\cdot 10^{-9}$ & \\
           \hline
    \end{tabular}
   \end{table}

\textbf{Example $5$.} In this example, we consider three real matrices $add32,epb1$ and $memplus$ which can be found in the Suite Sparse Matrix Collection \cite{DH}. These matrices are considered as a benchmark test. Some details on these matrices
are presented in Table \ref{tab:infomatrices}, including the condition number, and the sparsity of each matrix. The sparsity is defined as the ratio between the number of nonzero elements and the total number of elements, $n^2$. Results of the restarted-EBH, restarted-EBA methods and GE methods are stated in Table \ref{table:exampleShiftedSystemsSuite}. As indicated from Tables \ref{table:exampleShiftedSystemsFDM} and \ref{table:exampleShiftedSystemsSuite}, the restarted-EBH is much better in terms of the CPU times than the restarted-EBA. We also observe that the GE method requires highest CPU time than the restarted-EBH and the restarted-EBA methods.

\begin{table}[h]
    \centering
     \caption{matrix properties}
    \begin{tabular}{cccc}
    \hline
        Matrix & size $n$ & $Cond(A)$ & Sparsity  \\
        \hline
        $add32$ & $4960$ & $1.36\cdot 10^2$ & $8.0678\cdot 10^{-4}$ \\  
        $epb1$ & $14734$ & $5.94\cdot 10^3$ & $4.3785\cdot 10^{-4}$ \\  
        $memplus$ & $17758$ & $1.29\cdot 10^5$ &$3.1441\cdot 10^{-4}$ \\ 
        \hline
    \end{tabular}
    \label{tab:infomatrices}
\end{table}

\begin{table}[h]
\centering
  \caption{Example $5$:  Shifted solvers for nonsymmetric matrices from the Suite Sparse Matrix Collection matrices }
 \label{table:exampleShiftedSystemsSuite}
    \begin{tabular}{ccccccccc}
    \hline
      Matrix  & Iter. $(m)$ & \multicolumn{3}{c}{restarted-EBH}   &  \multicolumn{3}{c}{restarted-EBA} & GE \\
        &  & $\#$restarts & times$(s)$ & $R_{2m}$ & $\#$restarts & times$(s)$ & $R_{2m}$ & times$(s)$\\
         \hline
           $add32$ &  $5$ & $4$ & $1.95$ & $1.92\cdot 10^{-8}$ & $4$ & $2.55$ & $1.93\cdot 10^{-8}$ & $9.61$\\
        $n=4960$    &   $10$ & $2$ & $1.80$ & $9.36\cdot 10^{-8}$ & $2$ & $2.43$ & $9.29\cdot 10^{-9}$ \\
         \hline
               $epb1$ &  $5$ & $11$ & $7.67$ & $1.83\cdot 10^{-8}$ & $9$ & $9.07$ & $1.99\cdot 10^{-8}$& $82.19$\\
        $n=14734$    &   $10$ & $5$ & $8.06$ & $1.90\cdot 10^{-8}$ & $4$ & $10.56$ & $1.94\cdot 10^{-8}$& \\
         \hline
         $memplus$ &  $5$ & $23$ & $10.54$ & $1.98\cdot 10^{-8}$ & $20$ & $16.57$ & $1.99\cdot 10^{-8}$& $94.05$\\
        $n=17758$    &   $10$ & $10$ & $8.50$ & $1.99\cdot 10^{-8}$ & $7$ & $13.31$ & $1.99\cdot 10^{-8}$ &\\
           \hline
    \end{tabular}
   \end{table}

\section{Conclusion}\label{section:conclusion}
This paper presents the extended block Hessenberg method with its theoretical properties for the approximation of $f(A)V.$ We also gave algorithms based on the proposed method for solving  shifted linear systems with multiple right hand sides.  The numerical results show that the proposed method requires less CPU time, than the extended block Arnoldi method for functions and well-known benchmark matrices considered in all examples.

\end{document}